\documentclass[11pt,a4paper,reqno]{amsart}
\usepackage[
  a4paper,
  left=2.2cm,
  right=2.2cm,
  top=2.0cm,
  bottom=2.2cm
]{geometry}

\usepackage{arydshln}
\usepackage{mathrsfs}
\usepackage{amsmath}
\usepackage{amsfonts}
\usepackage{amssymb}
\usepackage[all,cmtip]{xy}
\usepackage[colorlinks]{hyperref}
\usepackage{color}
\usepackage{cleveref}
\usepackage{ulem}
\usepackage[abs]{overpic}
\usepackage{float}
\usepackage{enumitem}
\usepackage{graphicx}

\usepackage{amsthm}

\usepackage{etoolbox}

\def\<{\langle}
\def\>{\rangle}

\def\-{\overline}

\def\-{\overline}

\def\-{\overline}

\newtheorem{theorem}{Theorem}[section]
\newtheorem{lemma}[theorem]{Lemma}

\newtheorem{proposition}[theorem]{Proposition}

\newcommand\blfootnote[1]{%
  \begingroup
  \renewcommand\thefootnote{}\footnote{#1}%
  \addtocounter{footnote}{-1}%
  \endgroup
}

\title[Classification of homogeneous hypersurfaces in $\mathbb C^n$]
{\bf Classification of compact homogeneous strongly pseudoconvex hypersurfaces in $\mathbb C^n$}

\author{Hanlong Fang}
\address{School of Mathematical Sciences, Peking University, Beijing, 100871, China}
\email{hlfang$@$pku.edu.cn}

\subjclass[2010]{32C09, 32V40}
\keywords{Global embeddability of CR-manifolds in complex space, Homogeneous hypersurfaces}
\date{}

\begin{document}

\blfootnote{H. Fang is supported by National Key R\&D Program of China under Grant No.2022YFA1006700. }

\begin{abstract} 
By constructing explicit entire maps, we prove that every exceptional Morimoto--Nagano model $M_t^3$, $t>1$, can be realized as a compact real-analytic hypersurface in $\mathbb C^3$. This extends Isaev's realization results from restricted parameter ranges to all $t>1$ and provides a complete solution to a long-standing open problem in the Morimoto--Nagano classification of compact, simply connected, real-analytic hypersurfaces in complex Euclidean spaces that are homogeneous under the action of a connected Lie group of CR automorphisms.
\end{abstract}

\maketitle

\section{Introduction} 

This paper addresses the remaining CR-embeddability problem arising from the Morimoto--Nagano classification of compact, simply connected, real-analytic hypersurfaces in 
$\mathbb C^n$ that are homogeneous under the action of a connected Lie group of CR automorphisms.

Cartan's foundational work~\cite{Cartan1933,Cartan1932} classified, up to local biholomorphic equivalence, the locally homogeneous real-analytic Levi-nondegenerate hypersurfaces in $\mathbb C^2$. Morimoto and Nagano~\cite{MorimotoNagano1963} later studied the corresponding global classification problem for compact homogeneous hypersurfaces in $\mathbb C^n$. Outside the exceptional dimensions $3$ and $7$, their main theorem takes the following form.

\begin{theorem}[Morimoto--Nagano, 1963]
Let $S\subset\mathbb C^n$, $n\neq3,7$, be a compact, simply connected, real-analytic hypersurface.
If a connected Lie group acts transitively on \(S\) by CR automorphisms, then \(S\) is CR-equivalent to the unit sphere
\(S^{2n-1}\subset\mathbb{C}^n\).
\end{theorem}

In the exceptional dimensions 
$n=3,7$, the work of Morimoto and Nagano left open the possibility of exceptional models, namely the following one-parameter families of nonspherical hypersurfaces.  Denote the quadratic hypersurface in $\mathbb C^{n+1}$ by
\begin{equation}\label{eq:general-quadric}
Q^n:=\left\{z=(z_1,\ldots,z_{n+1})\in\mathbb C^{n+1}:          z_1^2+\cdots+z_{n+1}^2=1\right\},
\end{equation}
and, for $t>1$, set
\begin{equation}\label{eq:general-levels}
\begin{split}
&D_t^n:=\left\{z\in \mathbb C^{n+1}:|z_1|^2+\cdots+|z_{n+1}|^2<t\right\}\cap Q^n,\\
&M_t^n:=\partial D_t^n=\left\{z\in \mathbb C^{n+1}:|z_1|^2+\cdots+|z_{n+1}|^2=t\right\}\cap Q^n.   
\end{split}
\end{equation} 
Clearly, $M_t^n$ is diffeomorphic to a tangent sphere bundle over $S^n$, and hence admits a transitive action of $SO(n+1,\mathbb R)$. The hypersurfaces $M_t^n$
are known to be nonspherical~\cite{nemishafi}, and pairwise CR-nonequivalent~\cite{KaupZaitsev,BurnsHind}. The work of Morimoto and Nagano left open the question of whether \(M_t^3\) and \(M_t^7\) admit real-analytic CR embeddings into \(\mathbb C^3\) and \(\mathbb C^7\), respectively.

Isaev~\cite{Isaev2013} proved that every real-analytic CR embedding
$M_t^n\hookrightarrow\mathbb C^n$ extends to a biholomorphism from $D_t^n$ onto a domain in $\mathbb C^n$. Since the zero section
$S^n=Q^n\cap\mathbb R^{n+1}\subset D_t^n$ is totally real, the
restriction of the extended biholomorphism to $S^n$ is a  totally real embedding into $\mathbb C^n$. Using Smale's regular-homotopy theorem \cite{Smale} together with a homotopy-theoretic obstruction to totally real immersions, Gromov~\cite{Gromov1986} and Stout and Zame~\cite{StoutZame1986} proved that the seven-sphere $S^7$ admits no totally real embedding into $\mathbb C^7$. Isaev thus proved that no $M_t^7$, $t>1$, admits a real-analytic CR embedding into $\mathbb C^7$. 

In complex dimension three, Gromov~\cite{Gromov1973,Gromov1986} proved the existence of a real-analytic totally real embedding $S^3\hookrightarrow\mathbb C^3$ as an application of his $h$-principle. Ahern and Rudin~\cite{AhernRudin1985} constructed an explicit
polynomial totally real embedding. Forstneri\v{c}~\cite{Forstneric1986,Forstneric1990} later developed further constructions and generalizations. The holomorphic complexification of the Ahern--Rudin map is biholomorphic on a neighborhood of $S^3$ in $Q^3$. Since $D_t^3$ shrinks to $S^3$ as $t\rightarrow1$, this
gives embeddings for $t>1$ sufficiently close to $1$.

In 2013, Isaev \cite{Isaev2013} proved that the Ahern--Rudin map embeds $M_t^3$ for $1<t<1+10^{-6}$, and that the differential of the Ahern--Rudin map drops rank at some point of $M_t^3$ exactly when $t\geq\sqrt{(2+\sqrt2)/3}$. Isaev~\cite{Isaev2017} subsequently proved embeddability throughout $1<t<\sqrt{(2+\sqrt2)/3}$ for the Ahern--Rudin map. Using a different polynomial map, Isaev~\cite{Isaev2021} later enlarged the known range to $1<t<\frac{\sqrt5}{2}$.

In this paper, we combine the differential-equation method of Bell and Narasimhan for constructing holomorphic immersions~\cite{Bell}, Isaev's coordinate ansatz~\cite{Isaev2021}, and a natural one-parameter automorphism of \(Q^3\) to construct explicit holomorphic immersions \(Q^3\to\mathbb C^3\) that restrict to embeddings of \(M_t^3\) for every \(t>1\). Our main result is the following theorem.

\begin{theorem}\label{thm:main}
For every \(t>1\), there exist an open neighborhood \(V_t\) of
\(\overline{D_t^3}\) in \(Q^3\) and a holomorphic embedding
\[
\Phi_t:V_t\hookrightarrow\mathbb C^3
\]
that is the restriction of an explicitly given entire map
\(\mathbb C^4\to\mathbb C^3\). In particular,
\(\Phi_t|_{M_t^3}\) is a real-analytic CR embedding.
\end{theorem}

Together with the known non-embeddability of $M_t^7$ into
$\mathbb C^7$, Theorem~\ref{thm:main} completes the Morimoto--Nagano classification of compact homogenous strongly pseudoconvex hypersurfaces in $\mathbb C^n$. 
\begin{theorem}
Let $S\subset\mathbb C^n$ be a compact, simply connected, real-analytic hypersurface. If a connected Lie group acts transitively on $S$ by CR automorphisms, then exactly one of the following holds:
\begin{enumerate}
\item $S$ is CR-equivalent to the unit sphere
$S^{2n-1}\subset\mathbb C^n$;
\item $n=3$, and $S$ is CR-equivalent to $M_t^3$ for a unique $t>1$.
\end{enumerate}
\end{theorem}

We conclude this introduction by briefly outlining the main idea behind our construction. Combining the differential-equation method of Bell and Narasimhan with Isaev's coordinate ansatz, we obtain a local biholomorphism $F_0|_{Q^3}:Q^3\rightarrow\mathbb C^3$ by making the simplest linear choice for the exponent in the resulting differential equation. Remarkably, with this simple choice, the nontrivial fibers of $F_0|_{Q^3}$ admit an explicit description in terms of integer translations. We then observe a natural one-parameter family of fiberwise dilations $\delta_{\lambda}|_{Q^3}\in{\rm Aut}(Q^3)$, which appropriately force distinct points in each fiber of $F_{\lambda}|_{Q^3}:=(F_0|_{Q^3})\circ(\delta_{\lambda}|_{Q^3})$ to separate exponentially in Euclidean norm. A Cauchy--Schwarz estimate then shows that, for every fixed $t>1$, the map $F_{\lambda}|_{Q^3}$ is injective on a neighborhood of $\overline{D_t^3}$ for sufficiently large $\lambda$.
\bigskip

{\bf \noindent Acknowledgment}:
The author is deeply grateful to Professor Xiaojun Huang, for his lasting support and encouragement, and for generously sharing his mathematical insights over the years. In particular, he  thanks Professor Xiaojun  Huang for  bringing this problem to his attention even at his time as a graduate student at Rutgers University, for suggesting him to  study  Bell--Narasimhan's construction of holomorphic immersions by solving a differential equation, fully exploiting the algebraic structure, and constructing a sequence of immersions making use of the group structure of its automorphism with increasing injectivity radius. This turns out to be the crucial idea for our construction.  We also thank him for carefully reading the manuscript and offering valuable suggestions.

\section{The factorization of the explicit map}

For convenience, we make the following holomorphic change of
coordinates of $\mathbb C^4$:
\begin{equation*}
 w_1=z_1+iz_2,\quad w_2=z_1-iz_2,\quad
 w_3=z_3+iz_4,\quad w_4=z_3-iz_4.
\end{equation*}
In these coordinates, the quadric $Q^3$ becomes
\[
 \left\{W=(w_1,w_2,w_3,w_4)\in\mathbb C^4:     w_1w_2+w_3w_4=1\right\}.
\]
This is crucial  as it helps us  to find easily  a family of natural dilations that  blows up the distance of two points.
Consequently,
\begin{equation*}
D_t^3=\{W\in Q^3:\|W\|^2<2t\},\,M_t^3=\{W\in Q^3:\|W\|^2=2t\},\,{\rm where\,\,}\|W\|^2=\sum\nolimits_{j=1}^4|w_j|^2.
\end{equation*}

We begin with a brief review of the Bell--Narasimhan construction (see \cite[Page 19]{Bell}). Given an entire function $\psi$, choose an entire function $\varphi$ such that
\begin{equation}\label{dd}
d\cdot x^{d-1}-\varphi'(x)(1-x^d)=e^{\psi(x)},\quad x\in\mathbb C.
\end{equation}
Then, for any hypersurface
\begin{equation*}
X=\left\{
 (x,y^{(1)},\ldots,y^{(k)}):
 x^d+\sum\nolimits_{j=1}^k P_j(y^{(j)})=1
 \right\},\quad y^{(j)}\in\mathbb C^{n_j},\quad n_1+\cdots+n_k=:n,
\end{equation*}
where each \(P_j\) is homogeneous of degree \(d_j\), the holomorphic map $F:\mathbb C^{n+1}\to\mathbb C^n$ defined by
\begin{equation*}
 (x,y^{(1)},\ldots,y^{(k)})
 \longmapsto
(e^{\varphi(x)/d_1}y^{(1)},\ldots,
 e^{\varphi(x)/d_k}y^{(k)})
\end{equation*}
restrictions to a holomorphic immersion \(F|_X:X\to\mathbb C^n\).

As in \cite[(2.6), (2.15)]{Isaev2021}, we consider the holomorphic map $F_0:\mathbb C^4\to\mathbb C^3$ defined by
\begin{equation}\label{eq:F0}
F_0(w_1, w_2, w_3, w_4) =(w_1, w_3, w_2w_4\cdot h(w_3w_4)),
\end{equation}
where $h$ is an entire function. Setting \(s=w_3w_4\), we know from \cite[Lemma 2.1]{Isaev2021} that \(F_0|_{Q^3}\) is an immersion if and only if
\begin{equation*}
0\neq w_3\frac{\partial (w_2w_4 h)}{\partial w_2}-w_1\frac{\partial (w_2w_4 h)}{\partial w_4}
=-\frac{\partial \big((1-s)s\cdot h(s)\big)}{\partial s}
\quad\text{for all } W\in Q^3.
\end{equation*}
Motivated by \eqref{dd}, we impose
\begin{equation*}
-\frac{{\rm d} \big((1-s)s\cdot h(s)\big)}{{\rm d} s}=e^{\psi(s)}.
\end{equation*}
In the simplest case \(\psi(s)=\alpha s+\beta\),  solving for \(h\) yields
\begin{equation*}
h(s)=-\frac{e^\beta}{\alpha} \frac{e^{\alpha s}-1}{s(1-s)} \quad (\alpha\neq0)
\qquad\text{or}\qquad
h(s)=-\frac{e^\beta}{1-s} \quad (\alpha=0).
\end{equation*}
Since \(h\) is entire, we have \(\alpha=2\pi m i\) for some \(m\in\mathbb Z\setminus\{0\}\). In what follows, we fix $F_0$ with $h$ given by
\begin{equation}\label{eq:h}
 h(s)=\frac{e^{2\pi i s}-1}{s(1-s)}.
\end{equation}

Our second main ingredient is a natural one-parameter family of  automorphisms of $\mathbb C^4$ preserving $Q^3$, given by
\begin{equation}\label{eq:delta}
\delta_{\lambda}(w_1,w_2,w_3,w_4)=
\left(e^{\lambda s}w_1,e^{-\lambda s}w_2, e^{\lambda s}w_3,e^{-\lambda s}w_4\right), \,\,\lambda\in\mathbb R. 
\end{equation}
Note that $(\delta_{\lambda})^{-1}=\delta_{-\lambda}$. Now, we define $F_{\lambda}:=F_0\circ\delta_{\lambda}$, that is,
\begin{equation}
 F_{\lambda}(w_1,w_2,w_3,w_4)=
 \left(e^{\lambda s}w_1,e^{\lambda s}w_3,       w_2w_4\,e^{-2\lambda s}h(s)\right),
 \,\, s=w_3w_4.
\end{equation}

Theorem~\ref{thm:main} follows from the next proposition.
\begin{proposition}\label{prop:tube}
For every $\lambda>0$, the restriction of $F_{\lambda}$ to
\[
 U_{\lambda}=\{W\in Q^3:\|W\|^2<e^{\lambda}\}
\]
is a holomorphic embedding.
\end{proposition}

We will verify later that $F_0|_{Q^3}$ is locally one to one. As $F_{\lambda}|_{Q^3}$ is a holomorphic map from a three dimensional complex manifold to ${\mathbb C}^3$, it is locally biholomorphic. Since $\delta_{\lambda}|_{Q^3}$ is an
automorphism,
$F_{\lambda}|_{Q^3}$ is likewise
a local biholomorphism
for every $\lambda$. We proceed to prove the following core lemma, which gives an explicit description of the fibers of $F_{\lambda}|_{Q^3}$.
\begin{lemma}\label{prop:fibers}
If \(W\in Q^3\) satisfies \(w_1w_3=0\), then 
\begin{equation*}
(F_{\lambda}|_{Q^3})^{-1}(F_{\lambda}(W))=\{W\}.    
\end{equation*}
If $w_1w_3\neq0$, then \begin{equation*}
(F_{\lambda}|_{Q^3})^{-1}(F_{\lambda}(W))=\{W_m:m\in\mathbb Z\},    
\end{equation*} where
\begin{equation}\label{eq:Wm}
\begin{split}
W_m&=\left(e^{-\lambda m}w_1,
 e^{\lambda m}(w_2-\frac{m}{w_1}),
 e^{-\lambda m}w_3,
 e^{\lambda m}(w_4+\frac{m}{w_3})\right),\,\,m\in\mathbb Z.\\
\end{split}
\end{equation}
\end{lemma}
\begin{proof}[Proof of Lemma~\ref{prop:fibers}]
We first consider the case $\lambda=0$. Let $W'\in (F_0|_{Q^3})^{-1}(F_0(W))$ be arbitrary. Then $w_1'=w_1$ and $w_3'=w_3$. If $w_1=0$, then $w_3w_4=w_3'w_4'=1$, and hence $w_4'=1/w_3'=1/w_3=w_4$. Since the third component of $F_0$ equals $w_2 w_4 h(s)$ and $h(1)\neq0$, this determines $w_2'$. The case $w_3=0$ is identical, using $h(0)\neq0$.

If $w_1w_3\neq0$, then comparing the products of the three components of $F_0$ at $W$ and $W'$, namely 
\begin{equation*}
e^{2\pi iw_3w_4}-1=e^{2\pi iw_3'w_4'}-1,
\end{equation*}
we obtain $w_3'w_4'-w_3w_4=m\in\mathbb Z$. 

Now assume that $\lambda\in\mathbb R$ is arbitrary. Recall that $(F_{\lambda}|_{Q^3})^{-1}(F_{\lambda}(W))=\delta_{\lambda}^{-1}\left( (F_0|_{Q^3})^{-1}\bigl(F_0(\delta_{\lambda}(W))\bigr)\right)$. Applying the fiber formula to $\delta_{\lambda}(W)$ gives
$(F_0|_{Q^3})^{-1}\bigl(F_0(\delta_{\lambda}(W))\bigr) = \{\widehat W_m:m\in\mathbb Z\}$,
where
\[
 \widehat W_m:
 =
 \left(
 e^{\lambda s}w_1,\,
 e^{-\lambda s}(w_2-\frac{m}{w_1}),\,
 e^{\lambda s}w_3,\,
 e^{-\lambda s}(w_4+\frac{m}{w_3})
 \right).
\]
Since the axis fibers are already known to be singletons, we may assume \(w_1w_3\neq0\). For such points, the product of the $w_3$ and $w_4$ coordinates gives $e^{\lambda s}w_3\cdot e^{-\lambda s}(w_4+\frac{m}{w_3})=s+m$. Thus, when applying \(\delta_{\lambda}^{-1}=\delta_{-\lambda}\) to \(\widehat W_m\), the exponential factor is evaluated at \(s+m\) rather than at \(s\). Consequently, \(\delta_{\lambda}^{-1}(\widehat W_m)\) is given by \eqref{eq:Wm}, which completes the proof.
\end{proof}

\begin{proof}[Proof of Proposition~\ref{prop:tube}]
It remains to prove injectivity on $U_{\lambda}$. We argue by contradiction. Suppose that there exist distinct points $W,W'\in U_{\lambda}$ with $F_{\lambda}(W)=F_{\lambda}(W')$. By Lemma~\ref{prop:fibers}, we have $w_1w_3\neq0$. Notice that for any $m\in\mathbb Z$, we have $(W_m)_{-m}=W$, where
the points $W_m$, $(W_m)_{-m}$ are the preimages defined in \eqref{eq:Wm}.  Hence, we can assume that $W'=W_m$ for some $m>0$ (otherwise we can interchange the roles of $W$ and $W'$). Now formula \eqref{eq:Wm} yields the cross-pairing identity
\begin{equation*}
w_1(W_m)_2+w_3(W_m)_4=e^{\lambda m},
\end{equation*}
where $(W_m)_j$ denotes the 
$j$-th coordinate of $W_m$. By Cauchy--Schwarz,
\begin{equation*}
 e^{\lambda m}\leq \sqrt{|w_1|^2+|w_3|^2}          \sqrt{|(W_m)_2|^2+|(W_m)_4|^2}\leq \|W\|\,\|W_m\|<e^{\lambda},    
\end{equation*}
contradicting $m\geq1$.
Therefore $F_{\lambda}$ is injective on $U_{\lambda}$.
\end{proof}

\bibliographystyle{plain} 
\bibliography{refs}

\end{document}